\documentclass[12pt]{article}
\usepackage{amsmath,amssymb,amsthm}
\usepackage[margin=3cm]{geometry}
\usepackage{blindtext}
\usepackage{titlesec}
\title{Sections and Chapters}

\title{\bf Weighted group algebras}

\author{ Ali Rejali $^1$, \thanks{2020 Mathematics Subject Classifcation. Primary: 46J05; Secondary: 46J10} ,  \and Maryam Aghakoochaki  $^2$, \thanks{Corresponding author}}

\date{
	$^1$Department of Pure Mathematics, Faculty of Mathematics and Statistics, University of Isfahan, Isfahan 81746-73441, Iran \\ \texttt{rejali@sci.ui.ac.ir, Orcid: 0000-0001-7270-665X}\\%
Isfahan University  \\ \texttt{mkoochaki@sci.ui.ac.ir, Orcid: 0000-0002-3851-6550}	$^2$\\[2ex]%
\today
}

\newcommand{\tnorm}[1]{{\left\vert\kern-0.25ex\left\vert\kern-0.25ex\left\vert #1 
    \right\vert\kern-0.25ex\right\vert\kern-0.25ex\right\vert}}

\theoremstyle{plain}
\newtheorem{thm}{Theorem}[section]

\newtheorem{lem}[thm]{Lemma}
\theoremstyle{definition}
\newtheorem{rem}[thm]{\bf Remark}
\newtheorem{ex}[thm]{\bf Example}

\begin{document}
\maketitle

\begin{abstract}
		Let   $G$ be a locally compact Abelian group, and $w: G\to (0, \infty)$ be a Borel measurable weighted function.
In this paper, the algebraic and topological properties of group algebra are studied and assessed. 
We show that the weighted group algebra $L^{1}(G, w)$ is regular if and only if $w$ is a nonquasianalytic weight function. Also $L^{1}(G, w)$ is Tauberian, for any Borel measurable weight function $w$
on the group $G$.

\end{abstract}

\section{intrdoction}

In this paper, $G$ is an Abelian locally compact group, and $w: G\to (0, \infty)$ is a Borel measurable weight function.
Then $L^{1}(G, w)$ is the algebra of all  Borel measurable functions $f: G\to \mathbb C$ such that 
$$
\|f\|_{1, w} := \int_{G}\mid f(x)\mid w(x)dx< \infty
$$
The space   $L^{1}(G, w)$ with the multiplication 
$$
f\star g(x)= \int_{G}f(y)g(y^{-1}x)dy
$$
for $f,g \in  L^{1}(G, w)$, is a Banach algebra; see \cite{D}, \cite{DL} and \cite{RV}.
The first Arens product $\diamond$ on the second dual of a Banach algebra $A$ is defined as the following:
$$
<\Phi \diamond\Psi, f>= \Phi(\Psi f)
$$
Where
$$
\Psi f(a)= \Psi(fa)
$$
and
$$
fa(b)= f(ab)
$$
for all $\Phi, \Psi\in A^{**}$, $f\in A^{*}$ and $a,b \in A$. Assume that $A$ is a Banach algebra. $A$ is called Arens regular if for each $\Phi\in A^{**}$, the mapping $\Psi\mapsto \Phi o \Psi$ is weak*- continuous on $A^{**}$.
The Banach algebra $A$ is called regular, if for each proper subset $K$ of $\Delta(A)$
closed in the Gelfand topology and each point $\varphi\in\Delta(A)\backslash K$, there exists $f\in A$
such that $\hat{f}(\varphi) = 1$ and $\hat f$ vanishes on $K$. Assume that $E$ is a Banach $A$- bimodule. $A$ is called amenable if every continuous derivation $A\to E^{*}$
is inner. Let $A$ be a commutative Banach algebra with non-empty
character space. Put 
$$
A_{0}= \{a\in A: supp(\hat a) ~ is ~ compact\}
$$
If $A_{0}$ is norm dense in $A$, then $A$ is called Tauberian.
   Many authors studied algebraic and topological properties of $L^{1}(G, w)$. Gronback in \cite{Gr} showed that $L^{1}(G, w)$ is an amenable Banach algebra if and only if $G$ is an amenable group and
the weight function $w^{*}: G\to [1, \infty)$ is bounded, where 
$$
w^{*}(x):= w(x)w(x^{-1})  \quad (x\in G)
$$
 Rejali and Mehdipour in \cite{MER} proved that the Banach algebra $L^{1}(G, w)$ is Arens- regular if and only if $G$ is finite group or the function $\Omega: G\times G\to (0, 1]$ 
is 0- cluster function, this means that for all sequences $(x_{n})$ and $(y_{n})$ of  distinct members of $G$, there exist subsequences  $(x_{n}^{'})$ and $(y_{n}^{'})$ such that
\begin{align*}
\underset{n}{lim}\underset{m}{lim} \Omega(x_{n}^{'}, y_{n}^{'})=0= \underset{m}{lim}\underset{n}{lim} \Omega(x_{n}^{'}, y_{n}^{'})
\end{align*}
Where
$$
\Omega(x, y)= \frac{w(xy)}{w(x)w(y)}
$$
for $x,y\in G$. Rejali and Vishki in \cite{RV1} showed that the function $w^{*}$ is bounded if and only if there exists $\epsilon>0$ such that 
$$
\Omega(x, y)\geq\epsilon> 0 \quad (x,y\in G)
$$
As a result, $L^{1}(G, w)$ is Arens regular and amenable if and only if the group $G$ is finite. 

    Recently, many algebraic and homological properties of $L^{1}(G, w)$ and it's second dual were studied in the articles of Rejali and Mehdipour in [\cite{MER}, \cite{MER2} and \cite{MER3}].
 For further study, we refer readers to references \cite{DL},  \cite{Do}, \cite{GRS}, \cite{MR}, \cite{MR2},     \cite{R} and \cite{RV1} .
Domar in \cite{Do}, showed that $L^{1}(G, w)$ is Tauberian and regular if and only if the weight function $w$ is nonquasianalytic, that is
$$
\sum_{n=1}^{\infty}\frac{w(x^{n})}{n^{2}} <\infty
$$
for all $x\in G$. 
In this paper,  we show that $L^{1}(G, w)$ is always Tauberian. Therefore, $L^{1}(G, w)$ is regular if and only if the weight function $w$ is nonquasianalytic.
In \cite{kan},  there is another criterion for the regularity of $L^{1}(G, w)$. In fact,  $L^{1}(G, w)$ has the unique uniform norm property if and only if $L^{1}(G, w)$ is regular; see [\cite{kan}, 4. 7. 3].

  In this paper, we have studied briefly the algebraic and homological properties of $L^{1}(G, w)$ and show that each positive definite function on $L^{1}(G, w)$ is in the following form
$$
\varphi(f)= \int_{\widehat{G_{w}}}\gamma(f)d\mu(\gamma)
$$
for some $\mu\in M({\widehat{G_{w}}})$ and $f\in L^{1}(G, w)$, where ${\widehat{G_{w}}} = \Delta( L^{1}(G, w))$.
The correlation between  $ L^{1}(G, w)$ and $ L^{1}(G)$  are assesed.
Furthermore, we study the Gelfand representation of $ L^{1}(G, w)$. We show that the Gelfand map on $ L^{1}(G, w)$ is onto [resp. isometric] if $G$ is finite [resp. trival].

\section{ Weighted group algebras}

               Let $G$ be a locally compact group and $w: G\to (0, \infty)$ be a Borel measurable weighted function such that $w(xy)\leq w(x)w(y)$, for all $x, y\in G$. 
 Assume that $L^{\infty}(G, \frac{1}{w})$ is the space of all essentially weighted bounded function $f: G\to \mathbb C$ such that 
$$
\|f\|_{\infty, \frac{1}{w}}:= esssup\{\frac{\mid f(x)\mid}{w(x)}: x\in G\}<\infty
$$
Then $L^{\infty}(G, \frac{1}{w})$ with w- pointwise product,
$$
f\underset{w}{.}g(x):= \frac{f(x)g(x)}{w(x)}  \quad (x\in G)
$$
is a $C^{*}$- algebra; see\cite{RV}. Furthermore, the map 
\begin{align*}
\Phi ~: L^{\infty}(G, \frac{1}{w}) &\to L^{1}(G, w)^{*}\\
                                                      g &\mapsto \Phi_{g}
\end{align*}
Where
$$
\Phi_{g}(f):= \int_{G}f(x)g(x)dx   \quad (f\in L^{1}(G, w)).
$$
is isometrically isomorphism map.

\begin{thm}
Let $G$ be a locally compact group, and $w: G\to (0, \infty)$ be a Borel measurable weight function. Then
$$
 L^{1}(G, w)^{*}= L^{\infty}(G, \frac{1}{w}).
$$
 isometrically isomorphism as Banach spaces.
\end{thm}
\begin{proof}
Let $\Phi: L^{\infty}(G, \frac{1}{w})\to L^{1}(G, w)^{*}$ defined by  $g\mapsto \Phi_{g}$, where 
$$
\Phi_{g}(f):= \int_{G}f(x)g(x)dx   \quad (f\in L^{1}(G, w)).
$$
Then 
\begin{align*}
\mid \Phi_{g}(f)\mid &\leq \int_{G}\mid f(x)\mid w(x) \frac {\mid g(x)\mid}{w(x)}dx \\
                                 &\leq \|g\|_{\infty, \frac{1}{w}}. \|f\|_{1,w}
\end{align*}
Therefore 
$$
\|\Phi_{g}\|\leq \|g\|_{\infty, \frac{1}{w}}
$$
So $\Phi$ is continuous. Let $I\in  L^{1}(G, w)^{*}$ and for $f\in L^{1}(G)$, define 
$J(f):= I(fw)$. Then $J\in L^{1}(G)^{*}= L_{\infty}(G)$. Hence there exist $g\in L_{\infty}(G)$ such that 
$$
J(f)= \int_{G}f(x)g(x)dx
$$
for all $f\in L^{1}(G)$. Let $h:= gw$. Then $h\in  L^{\infty}(G, \frac{1}{w})$ and, 
\begin{align*}
\Phi_{h}(k) &= \int_{G}h(x)k(x)dx\\
                   &= \int_{G}g(x)w(x)k(x)dx\\
                    &= J(kw)= I(k)   
\end{align*}
for all $k\in  L^{1}(G, w)$. Thus $\Phi_{h}= I$ and $\Phi$ is surjective. Also, 
\begin{align*}
\|\Phi_{g}\| &= sup \{\mid \Phi_{g}(f)\mid : \|f\|_{1,w}\leq1\}\\
                   &= sup \{\mid \varphi_{\frac{g}{w}}(fw)\mid: \|fw\|_{1}\leq 1\}\\
                    &= \|\frac{g}{w}\|_{\infty}= \|g\|_{\infty, \frac{1}{w}}
\end{align*}
for each $g\in  L^{\infty}(G, \frac{1}{w})$.
Where $\varphi: L^{\infty}(G)\to L^{1}(G)^{*}$ defined by $f\mapsto \varphi_{f}$ such that 
$$\varphi_{f}(h)= \int_{G}f(x)h(x)dx$$
for each $f\in  L^{\infty}(G)$, $h\in L^{1}(G)$. Thus $\Phi$ is an isometric map.
\end{proof}
Let $w: G\to [1, \infty)$ be a Borel- measurable weighted function. Then $M(G, w)$ is the Banach algebra of all bounded Radon measures $\mu: B(G)\to \mathbb C$ such that 
$\mu w\in M(G)$ where 
$$
\mu w(E) = \int_{E}w(x)d\mu(x)   \quad (E\in B(G))
$$
If $\mu, \nu\in M(G, w)$, then
$$
\mu\star\nu(E)= \int_{G}\int_{G} \chi_{E}(xy)d\mu(x)d\nu(y)
$$
and 
$$
\|\mu\|_{w}:= \|\mu w\|.
$$
Since $w\geq 1$, so  $M(G, w)$  is norm- dense in $M(G)$ and 
$$
 M(G, w)= C_{0}(G, \frac{1}{w})^{*}
$$
isometrically isomorphism as Banach algebras. Where $C_{0}(G, \frac{1}{w})$ is the algebra of  all  Borel- measurable functions $f: G\to \mathbb C$ such that $\frac{f}{w}\in C_{0}(G)$; see \cite{RV}.
\begin{thm}
Let $G$ be a locally compact group, and $w: G\to (0, \infty)$ be a Borel measurable weighted function. Then
$$
 M(G, w)= C_{0}(G, \frac{1}{w})^{*}
$$
under the duality
\begin{align*}
\Phi:    M(G, w) &\to  C_{0}(G, \frac{1}{w})^{*}\\
                         & \mu\mapsto \Phi_{\mu}
\end{align*}
where
$$
\Phi_{\mu}(f)= \int_{G}f(x)d\mu(x)   \quad (f\in C_{0}(G,  \frac{1}{w}))
$$
 isometrically isomorphism as Banach algebras.
\end{thm}
\begin{proof}
Let $f\in  C_{0}(G, \frac{1}{w})$ and $\mu\in M(G, w)$. Then
$$
\mid \Phi_{\mu}(f)\mid \leq \|f\|_{\infty, \frac{1}{w}}. \|w\mu\|_{w}
$$
Thus $\|\Phi_{\mu}\|\leq\|\mu\|$, so $\Phi$ is continious . Assume that $I, J\in  C_{0}(G, \frac{1}{w})^{*}$. The following is the yield:
$$
I\star J(f):= \int_{G}\int_{G} f(xy)d\mu(x)d\nu(y)
$$
where $\mu$ [resp. $\nu$] is the corresponding measure to the linear function I [resp. J]. It is to be noted that  $ M(G, w)= C_{0}(G, \frac{1}{w})^{*}$ as Banach spaces; see \cite{RV}. Therefore 
\begin{align*}
\Phi(\mu\star\nu)(f) &= \int_{G} f(z)d\mu\star \nu(z)\\
                                 &=  \int_{G}\int_{G} f(xy)d\mu(x) d\nu(y)\\
                                 &= \Phi_{\mu}\star \Phi_{\nu}(f)
\end{align*}
Hence $\Phi$ is a homomorphism. Also,
 \begin{align*}
\|\Phi_{\mu}\| &= sup\{ \mid\Phi_{\mu}(f)\mid: \|f\|_{\infty, \frac{1}{w}}\leq 1\}\\
                        &=  sup\{ \mid\varphi_{\mu w}(\frac{f}{w})\mid: \|\frac{f}{w}\|_{\infty}\leq 1\}\\
                        &= \|\mu w\|= \|\mu\|_{w}
\end{align*}
where $\varphi: M(G)\to C_{0}(G)^{*}$ such that $\nu\mapsto \varphi_{\nu}$ and $\varphi_{\nu}(g)= \int_{G}g(x)d\nu(x)$, for $g\in C_{0}(G)$.
So $\Phi$ is an isometric map. 
If $J\in C_{0}(G, \frac{1}{w})^{*}$, then $I(g):= J(gw)$, for $g\in C_{0}(G)$, defined a linear functional in $ C_{0}(G)^{*}= M(G)$. Thus there exist unique $\nu\in M(G)$ such that 
$$
I(g)= \int_{G}g(x)d\nu(x).
$$
Put $\mu = \nu w^{-1}\in M(G, w)$. Then the following is immediate, 
\begin{align*}
\Phi_{\mu}(f) &= \int_{G}f(x)d\mu(x)\\
                      &= \int_{G} \frac{f(x)}{w(x)}d\nu(x)\\
                      &=I(\frac{f}{w})= J(f)
\end{align*}
for each $f\in C_{0}(G, \frac{1}{w})$. Thus $\Phi_{\mu}= J$, so $\Phi$ is surjective. This completes the proof.
\end{proof}
In this paper, $\widehat{G_{w}}$ is the set of all Borel- measurable multiplicative complex-valued functions $\chi$ such that $\frac{\chi}{w}$ is bounded. In the following, we modify the result [\cite{kan}, 2.8.2], for easy reference.
\begin{thm}
Let $G$ be a locally compact Abelian group, and $w: G\to (0, \infty)$ be a Borel measurable weighted function. Then 
\begin{align*}
\Phi : \widehat{G_{w}} &\to \Delta(L^{1}(G, w))\\
                                      & \chi\mapsto \Phi_{\chi}
\end{align*}
where
$$
 \Phi_{\chi}(f):= {\hat f}(\chi)
$$
is homomorphic and isomorphism.
\end{thm}
\begin{proof}
(i) If $\chi\in \widehat{G_{w}}$, then 
\begin{align*}
 \Phi_{\chi}(f_{1}\star f_{2}) &= (f_{1}\star f_{2}\widehat)(\chi)\\
                                                &= (\widehat{f_{1}}.\widehat{f_{2}})(\chi)\\
                                                &=  \widehat{f_{1}}(\chi).\widehat{f_{2}}(\chi)
\end{align*}
for all $f_{1},f_{2}\in L^{1}(G, w)$. Thus $ \Phi_{\chi}$ is multiplicative. Since $ L^{1}(G, w)$ is a commutative Banach algebra, so $ \Phi_{\chi}$ is  continuous and 
\begin{align*}
\mid \Phi_{\chi}(f)\mid &= \mid\int_{G}\overline{\chi}(x)f(x)dx\mid\\
                                    &\leq \|\chi\|_{\infty, \frac{1}{w}}.\|f\|_{1,w}
\end{align*}
for all $f\in L^{1}(G, w)$. Therefore
$$
\|\Phi_{\chi}\|\leq \|\chi\|_{\infty, \frac{1}{w}}.
$$
(ii) Due to [\cite{kan}, 2. 8. 10], $ L^{1}(G, w)$ is semisimple. Hence the Gelfand map\\
 $\Gamma_{w}: L^{1}(G, w)\to C_{0}(\widehat{G_{w}})$ where $f\mapsto \hat f$ such that 
$$
\hat{f}(\chi)= \int_{G} \overline{\chi(x)}f(x)dx
$$
 is injective. If $\Phi_{\chi_{1}}= \Phi_{\chi_{2}}$, then 
$$
\hat{f}(\chi_{1})= \hat{f}(\chi_{2})
$$
for all $f\in L^{1}(G, w)$. Thus $\chi_{1}= \chi_{2}$ and so $\Phi$ is injective.\\
(iii) Assume that $\varphi\in \Delta( L^{1}(G, w))$ and $g\in  L^{1}(G, w)$, where $\varphi(g)=1$. If
$$
\chi(y)= \int_{G}\overline{\varphi(x)}.\overline{g(y^{-1}x)}dx
$$
then $\varphi= \Phi_{\chi}$ where $\chi\in\widehat{G_{w}}$ and $\Phi$ is surjective.\\
(iv) If $\chi_{\alpha}\to \chi$, then
$$
\Phi_{\chi_{\alpha}}(f)= {\hat f}(\chi_{\alpha})\to {\hat f}(\chi)= \Phi_{\chi}(f)
$$
for all $f\in L^{1}(G, w)$. Thus $\Phi_{\chi_{\alpha}}\overset{w^{*}}{\to} \Phi_{\chi}$ and conversly.
\end{proof}


\section{Regularirty of weighted group algebra}

                        Let $G$ be a locally compact Abelian group, and $w: G\to (0, \infty)$ be a Borel measurable weight function.
Assume that $A= ( L^{1}(G, w), \star)$. Then $A$ is a commutative semisimple Banach algebra. Define
\begin{align*}
r_{A}(f) & =inf \{\|f^{n}\|_{1,w}^{\frac{1}{n}} : n\in \mathbb N\}\\
                         &= \underset{n\to \infty}{lim} \{\|f^{n}\|_{1,w}^{\frac{1}{n}}\}     \quad (f\in A)
\end{align*}
Then\\
(i) 
$$
r_{A}(\lambda f) = \overline{ \lambda} r_{A}(f)
$$
(ii)
$$
r_{A}(f + g) \leq r_{A}(f)+ r_{A}(g)
$$
(iii)
$$
r_{A}(f\star g)\leq r_{A}(f)r_{A}(g)
$$
(iv)
$$
\mid r_{A}(f) - r_{A}(g)\mid\leq r_{A}(f-g)
$$
(v) 
$$
r_{A}(f)\leq \|f\|_{1,w}
$$
for all $f, g\in A$ and $\lambda\in\mathbb C$; see [\cite{kan}, 1.2.13].\\
Let 
$$
I= Ker(r_{A}) =\{ f\in L^{1}(G, w): r_{A}(f)= 0\}
$$
Then $I$ is a closed ideal in $L^{1}(G, w)$. Thus $A_{r}:= (\frac{A}{I}, \|.\|)$ is a commutative semisimple Banach algebra, where 
$$
\|f+I\|= inf\{\|f+ g\|_{A}: r_{A}(g)= 0\}
$$
Put 
$$
\|f\|_{r, w}:= \|f+ Ker(r_{A})\|_{\frac{A}{I}}
$$
for $f\in L^{1}(G, w)$. Thus 
$(A, \|.\|_{r, w})$  is a commutative semisimple Banach algebra. As a result $ \|.\|_{r, w}$ is equivalent with $ \|.\|_{1, w}$; see [\cite{kan}, 2.1.11]. Clearly
$$
\|f\|_{ {r, w}}\leq \|f\|_{ {1, w}}
$$
and there exists some $M>0$ such that
$$
M \|f\|_{ {1, w}}\leq \|f\|_{ {r, w}}
$$
for all $f\in L^{1}(G, w)$. Let $E$ be an  algebra and $\mid.\mid: E\to \mathbb C$ be a submultiplicative norm on $E$ such that $\mid x^{2}\mid= \mid x\mid^{2}$, for all $x\in E$. Then $\mid .\mid$ is called uniform norm; see [\cite{kan}, 4.6.1].
Let $A$ be a commutative  Banach algebra and $\mid .\mid_{A}$ be a uniform norm on $A$. Then
$$
\mid a\mid_{A}\leq r_{A}(a):= \underset{n\to \infty}{lim}\|a^{n}\|^{\frac{1}{n}}
$$
for all $a\in A$. Define
$$
E:= \{ \varphi\in \Delta(A): \mid\varphi(x)\mid\leq \mid x\mid_{A}, ~ for ~ x\in A\} 
$$
for all $x\in A$; [\cite{kan}, 4.6.2].
Let $A$ be a commutative semisimple Banach algebra. Then $A$ admits a uniform norm and $x\mapsto r_{A}(x)$ is a uniform norm on $A$; see [\cite{kan}, 4.6.3].
Assume that $a\in A$, so 
$$
\mid a\mid_{A} \leq r_{A}(a) = \|\hat a\|_{\infty}
$$
for each uniform norm $\mid.\mid_{A}$.\\
The Banach algebra $A$ has the unique uniform norm property, if $ r_{A}$ is the only uniform norm on $A$.

  Let $\mid.\mid_{A}$ be a uniform norm on the algebra $A$ and $B$ be the compilation of $A$. Then 
\begin{align*}
r_{A}(a) &= r_{B}(b)\\
                           &= \underset{n\to\infty}{lim}\mid a^{2^{n}}\mid_{A}^{\frac{1}{2^{n}}}\\
                          &= \mid a\mid_{A}  
\end{align*}
for each $a\in A$.
 \begin{thm}[\cite{kan}, 4.7.3]\label{tnr}
Let $G$ be a locally compact Abelian group, and $w: G\to (0, \infty)$ be a Borel measurable weighted function. Then $ L^{1}(G, w)$ has the unique uniform norm property, if and only if 
$ L^{1}(G, w)$  is regular.
\end{thm} 
In the following Lemma, the correlation between $r_{ L^{1}(G, w)}$ and $r_{ L^{1}(G)}$ are assessed.
\begin{lem}
Let $G$ be a locally compact Abelian group,  $w: G\to (0, \infty)$ be a Borel measurable weight function, and $f: G\to \mathbb C$ be a Borel measurable function with compact support. If $f\in L^{1}(G, w)\cap L^{1}(G)$,
then there exist $M>0$  such that 
$$
\frac{1}{M}r_{ L^{1}(G)}(f)\leq r_{ L^{1}(G, w)}\leq Mr_{ L^{1}(G)}(f)
$$
\end{lem}
\begin{proof}
If $K:= supp(f)$, then $ supp(f^{n})\subseteq K^{n}$ is compact. Due to [\cite{kan}, 1.3.3], $w$  is bounded on $K$. Assume that 
$$
M:= sup\{w(x): x\in K\cup K^{-1}\}
$$
Then for all $x\in G$, we have 
$$
w(e)\leq w(x)w(x^{-1})
$$
and so
$$
w(x)\geq\frac{w(e)}{w(x^{-1})}
$$
If $x= x_{1}x_{2}\cdots x_{n}\in K^{n}$, then $x^{-1}= x_{1}^{-1}x_{2}^{-1}\cdots x_{n}^{-1}$ and so 
$$
w(x)\leq  w(x_{1})\cdots w(x_{n})\leq M^{n}
$$
and 
$$
w(x^{-1})\leq  w(x_{1}^{-1})\cdots w(x_{n}^{-1})\leq M^{n}
$$
for all $n\in\mathbb N$. As a result 
$$
\frac{w(e)}{M^{n}}\leq w(x)\leq M^{n}
$$
for all $x\in K^{n}$. Therefore 
\begin{align*}
\frac{w(e)}{M^{n}}\|f^{n}\|_{1} &\leq \|f^{n}\|_{1,w}\\
                                                        &= \int_{G}\mid f^{n}(x)\mid w(x)dx\\
                                                        &= \int_{K^{n}}\mid f^{n}(x)\mid w(x)dx\leq M^{n}\|f^{n}\|_{1}
\end{align*}
As a result 
\begin{align*}
\frac{w(e)^{\frac{1}{n}}}{M}\|f^{n}\|_{1}^{\frac{1}{n}} \leq \|f^{n}\|_{1, w}^{\frac{1}{n}}\leq M\|f^{n}\|_{1}^{\frac{1}{n}}
\end{align*}
Therefore
$$
\frac{1}{M}r_{L^{1}(G)}(f)\leq r_{L^{1}(G, w)}(f)\leq M r_{L^{1}(G)}(f).
$$
\end{proof}
 For the definition of $hk$- topology, we refer the reader to \cite{kan}.
\begin{thm}
Let $G$ be a locally compact Abelian group, and $w: G\to [1, \infty)$ be a Borel measurable weighted function such that
 $w$ is a non- quasianalytic weight function. Then\\
(i) 
$$
w^{*}-cl(\Delta(L^{1}(G, w)))= \Delta(M(G, w)).
$$
(ii)
$$
w^{*}-cl(L^{1}(G, w))= M(G, w).
$$
\end{thm}
\begin{proof}
(i) Since $L^{1}(G, w)$ is a commutative semisimple Banach algebra with bounded approximate identity and  $L^{1}(G, w)$  is a closed ideal of 
$$
M(G, w)= { M}(L^{1}(G, w))
$$
then by applying [\cite{kan}, 4.1.10], we have
$$
hk- cl(\Delta(A))= \Delta(M(A))
$$
where $A= L^{1}(G, w)$. Since $A$ is regular, so $\tau_{hk}= \tau_{w^{*}}$. As a result due to \cite{DL} we have
\begin{align*}
w^{*}-cl (\widehat{G_{w}}) &= hk-cl (\Delta(L^{1}(G, w)))\\
                                               &= \Delta(M(L^{1}(G, w))= \Delta(M(G, w)).
\end{align*}
Where
$$
 M(L^{1}(G, w)):= \{\sigma\in C_{b}(\Delta(A): ~ \sigma. \hat A\subseteq \hat A\}.
$$
(ii) Since  $L^{1}(G, w)$ is a closed ideal of 
$$
M(G, w)= C_{0}(G,\frac{1}{w})^{*}
$$
and 
$$ 
w^{*}-cl(L^{1}(G))= M(G)
$$
So 
$$
w^{*}-cl(L^{1}(G, w))= M(G, w).
$$
\end{proof}

\begin{lem}\label{l3}
Let $G$ be a locally compact Abelian group, and $w: G\to (0, \infty)$ be a Borel measurable weighted function. Then \\
(i) 
If $A= l_{1}(G,w)$, then
$$
r_{A}(x)= \underset{n\to\infty}{lim}w(x^{n})^{\frac{1}{n}}:= r_{w}(x)
$$
for all $x\in G$.\\
(ii)
If $\chi\in\widehat{G_{w}}$, then 
$$
{r_{w}(x^{-1})}^{-1}\leq \mid\chi(x)\mid\leq r_{w}(x)
$$
for all $x\in G$.\\
(iii)
$$
0<\mid\chi(x)\mid\leq w(x)
$$
for all $x\in G$.
\end{lem}
\begin{proof}
(i) If $x\in G$, then $\delta_{x}\in A$ and 
$$
\|\delta_{x}\|_{1, w}= w(x)
$$
Therefore by applying [\cite{kan}, 1.2.5] we have
\begin{align*}
r_{A}(x) &= \underset{n\to\infty}{lim}\|\delta_{x}^{n}\|_{A}^{\frac{1}{n}}\\
                &= \underset{n\to\infty}{lim}\|\delta_{x^{n}}\|_{1, w}^{\frac{1}{n}}\\
                &= \underset{n\to\infty}{lim} w(x^{n})^{\frac{1}{n}}
\end{align*}
(ii) If $\chi\in \widehat{G_{w}}$, then there exists $M>0$ such that 
$$
\mid\chi(x)\mid\leq M w(x)
$$
for all $x\in G$. This implies that 
$$
\mid\chi(x)\mid^{n} = \mid\chi(x^{n})\mid \leq Mw(x^{n})
$$
for all $n\in \mathbb N$ and so
 $$
\mid\chi(x)\mid\leq M^{\frac{1}{n}}w(x^{n})^{\frac{1}{n}}
$$
As a result
$$
\mid\chi(x)\mid\leq  r_{w}(x)
$$
Moreover
$$
\frac{1}{\mid\chi(x)\mid}= \mid\chi(x^{-1})\mid\leq r_{w}(x^{-1})
$$
Therefore 
$$
\mid\chi(x)\mid\leq  r_{w}(x^{-1})^{-1}
$$
for all $x\in G$.\\
(iii)
$$
w(x^{n})\leq w(x)^{n}
$$
for all $n\in \mathbb N$.  Therefore
$$
r_{w}(x):= \underset{n\to\infty}{lim}w(x^{n})^{\frac{1}{n}}\leq w(x).
$$
\end{proof}

\begin{lem}\label{gsm}
Let $G$ be a locally compact Abelian group, and $w: G\to (0, \infty)$ be a Borel measurable weighted function. Then\\
(i)
$(\widehat{G_{w}}, \underset{w}{.})$ is semigroup if and only if $w$ is multiplicative.\\
(ii) $(\widehat{G_{w}}, \underset{w}{.})$ is group if and only if $w\equiv1$.\\
(iii) $ \widehat{G_{w}}= \widehat G$ if and only if 
$
\underset{n\to\infty}{lim}w(x^{n})=1
$
for all $x\in G$.
\end{lem}
\begin{proof}
(i)
If $\chi\in \widehat{G_{w}}$, then $\chi$ is multiplicative and $\chi  \underset{w}{.} \chi\in \widehat{G_{w}}$, then
\begin{align*}
\chi  \underset{w}{.} \chi(xy) &= \frac{\chi(xy)^{2}}{w(xy)}\\
                                                &=  \frac{\chi(x)^{2}\chi(y)^{2}}{w(xy)}
\end{align*}
Moreover 
\begin{align*}
\chi  \underset{w}{.} \chi(xy) &= \chi  \underset{w}{.} \chi(x)\chi  \underset{w}{.} \chi(y)\\
                                               &=  \frac{\chi(x)^{2}\chi(y)^{2}}{w(x)w(y)}
\end{align*}
for $x,y \in G$. In the other hand, according to part (iii) of Lemma \ref{l3}  $\chi$ is nonzero. This implies that
 $w(xy)= w(x)w(y)$. As a result, $w$  is multiplicative.\\
 Conversely, it is obvious.\\
(ii) If $\widehat{G_{w}}$  is group, then $\widehat{G_{w}}$ is semigroup and so by applying part (i), $w$ is multiplicative. Thus $w\in \widehat{G_{w}}$ is identity. Assume that $\chi\in \widehat{G_{w}}$, so there exists $\chi^{-1}\in \widehat{G_{w}}$
such that 
$$
\chi  \underset{w}{.} \chi^{-1}= w
$$
In another hand
$$
r_{w}(x^{-1})^{-1}\leq \mid\chi(x)\mid\leq r_{w}(x)
$$
for all $x\in G$. Since $w$ is multiplicative, therefore 
$$ 
r_{w}(x)= w(x)
$$
and 
$$
r_{w}(x^{-1})= \frac{1}{w(x)}
$$
then
$$
\mid\chi(x)\mid= w(x)
$$
and 
$$
w(x)= \mid\chi^{-1}(x)\mid
$$
As a result $w^{2}(x)=1$ and so $w(x)= 1$, for all $x\in G$.

   Conversely, it is obvious.\\
(iii) If $\chi\in \widehat{G_{w}}$, then 
$$
r_{w}(x^{-1})^{-1}\leq \mid\chi(x)\mid\leq r_{w}(x)
$$
for all $x\in G$. Since $\mid\chi(x)\mid= 1$, so 
$$
r_{w}(x^{-1})^{-1}\leq 1\leq r_{w}(x)
$$
Thus
$$
\underset{n\to\infty}{lim}w(x^{n})^{\frac{1}{n}}= r_{w}(x)=1
$$
Conversely,  by applying [\cite{kan}, 2.8.3] it is proved.
\end{proof}
Let $G$ be an Abelian group. Then $L^{1}(G)$ is regular; see[\cite{kan}, 4. 4.14].
In the following Theorem, by a similar argument, we show that $L^{1}(G, w)$  is regular.
\begin{thm}\label{l1re}
Let $G$ be a locally compact Abelian group, and $w: G\to [1, \infty)$ be a Borel measurable weighted function such that $w(x^{n})^{\frac{1}{n}}= 1$. Then  $ L^{1}(G, w)$ is regular.
\end{thm}
\begin{proof}
By hypothesis, we have $\widehat{G_{w}}= \widehat G$.
If $E\subseteq \widehat{G_{w}}$ is $w^{*}$- closed and $\varphi\notin E$, then $W=  \widehat{G_{w}}\backslash E$ is open and $\varphi\in W$. Since the multiplication is continuous, so there exist open sets $U$ and $V$ such that $\varphi\in U$, 
$1_{G}\in V$ and $U.V\subseteq W$. Thus $UV\cap E= \emptyset$. Assume that $u,v\in C_{00}(G)^{+}$ where $supp(\hat u)\subseteq U$ and $supp(\hat v)\subseteq V$, then $f=uv\in L^{1}(G, w)$ and 
$$
\widehat{uv}= \hat u\star \hat v
$$
As a result
\begin{align*}
supp(\hat f) &= supp(\hat u\star \hat v)\\
                    &\subseteq supp(\hat u)supp(\hat v)\\
                     & \subseteq UV
\end{align*}
This implies that $\hat{f}\mid_{E}= 0$ and ${\hat f}(\varphi)\neq 0$.
\end{proof}

\begin{lem}\label{lT}
Let $G$ be an Abelian locally compact group and   $w: G\to (0, \infty)$ be Borel- measurable weighted function. Then $L^{1}(G, w)$  is Tauberian.
\end{lem}
\begin{proof}
$L^{1}(G, w)$ has a bounded  approximate identity in $C_{00}(G)$. Thus by Kohen's factorization Theorem, we have
$$
{L^{1}(G, w)}\star L^{1}(G, w)= L^{1}(G, w)
$$
Hence if $f\in L^{1}(G, w)$, then there exist $g,h \in L^{1}(G, w)$ such that $f= g\star h$. This implies that, if $g,h\in L^{1}(G, w)$, where $g_{n}\to g$ and $h_{n}\to h$ in $C_{00}(G)$, then
$g_{n}\star h_{n}\to g\star h$ and $g\star h\in C_{00}(G)$. Thus
$$
supp(\widehat{g_{n}\star h_{n}} )= supp(\widehat{g_{n}}.\widehat{h_{n}})\subseteq supp(\widehat{h_{n}})
$$
is compact. Therefore for all $\epsilon>0$ there exist $m$ such that 
$$
\|g_{m}\star h_{m}-f\|_{1,w}< \epsilon
$$
and  $k= g_{m}\star h_{m}$ in $C_{00}(G)$ so in $ L^{1}(G, w)$ where $\|f-k\|_{1,w}< \epsilon$.\\
 As a result, $C_{00}(G)\subseteq  L^{1}(G, w)_{0}$. This completes the proof.
\end{proof}
\begin{rem}
Let $G$ be a locally compact Abelian group and $w: G\to (0, \infty)$ be a Borel measurable weighted function. If $ L^{1}(G, w)$ is  Tauberian and 
$$
\tau_{w^{*}}= \tau_{hk}
$$
on $ L^{1}(G, w)$.
Therefore we equip $\widehat{G_{w}}$ with $hk$- topology and $ L^{1}(G, w)$ with $w^{*}$- topology on $\Delta( L^{1}(G, w))$.
\end{rem}
Let $w: G\to [1, \infty)$ be a Borel measurable weighted function on a locally compact Abelian group $G$, such that
$$
\sum_{n=-\infty}^{+\infty}\frac{Ln w(x^{n})}{1+n^{2}}< \infty \quad (x\in G)
$$
Then $w$ is said non-quasianalytic weight function on locally compact Abelian group. Thus $ L^{1}(G, w)$  is regular; see [\cite{kan}, 4. 7. 11].
If  $w: G\to [1, \infty)$ is a  Borel- measurable weight, then $L^{1}(G, w)$ is regular and Tauberian if and only if $w$ is nonquasianalytic weight; see [{\cite{LN}}, p.447].
By using Lemma \ref{lT}, the following theorem is immediate.
\begin{thm}\label{tln}
Let $G$ be an Abelian locally compact group and   $w: G\to [1, \infty)$ be Borel- measurable weighted function. Then  $L^{1}(G, w)$  is regular if and only if $w$ is nonquasianalytic weight.
\end{thm}

\begin{thm}\label{thrd}
Let $w: G\to [1, \infty)$ be a Borel measurable weighted function on a locally compact Abelian group $G$. Then the following expressions are equivalent:\\
(i) The algebra $ L^{1}(G, w)$  is regular.\\
(ii) The weight $w$ is  nonquasianalytic.\\
(iii)
 $$
r_{w}(x)=  lim w(x^{n})^{\frac{1}{n}}=1
$$
for all $x\in G$.\\
(iv) 
$$
\Delta( L^{1}(G, w))= \Delta( L^{1}(G))
$$
\end{thm}
\begin{proof}
Due to \cite{PR} and Theorem \ref{tln} part (i) is equivalent to part (ii).\\
Kaniuth in [\cite{kan}, 4. 7. 5] showed that $(ii) \to (iii)$.\\
Due to Theorem \ref{l1re}, part (iii) concludes part (i).\\
By applying part (iii) of Lemma \ref{gsm}, it is proved that part (iii) is equivalent to part (iv). 

\end{proof}


\section{Gelfand representation of $L^{1}(G, w)$}
In this section, we examine the algebraic and topological properties of the Gelfand mapping on $L^{1}(G, w)$, for further reading,
refer to refrence \cite{GRS}.
\begin{lem}
Let $G$ be a locally compact Abelian group,  $w: G\to (0, \infty)$ be a Borel measurable weighted function, and $\varphi: L^{1}(G, w)\to \mathbb C$ is positive definite. Then there exist unique $\mu\in M(\widehat{G_{w}})$ such that
$$
\varphi(f)= \int_{\widehat{G_{w}}}\gamma(f)d\mu(\gamma).
$$
for all $f\in L^{1}(G, w)$.
\end{lem}
\begin{proof}
If $S= ((L^{1}(G, w), \star)$, then $S$ is a $\star$- semigroup. Thus by applying Bochner's theorem for semigroups there exists $\mu\in M({\widehat S})$ such that 
$$
\varphi(f)= \int_{\widehat S}\gamma(f)d\mu(\gamma)
$$ 
for all $f\in  L^{1}(G, w)$; see \cite{PR}. Since $\chi$ is multiplicative and 
\begin{align*}
\widehat{S} &= \{\chi: S\to\mathbb C: \chi ~ is ~ multiplicative\}\\
                    &= \Delta( L^{1}(G, w))= {\widehat{G_{w}}}.
\end{align*}
Therefore
$$
\varphi(f)= \int_{\widehat{G_{w}}}\gamma(f)d\mu(\gamma).
$$
\end{proof}

\begin{lem}
Let $G$ be a locally compact Abelian group, and  $w: G\to (0, \infty)$ be a Borel measurable weighted function. Then \\
(i) If $f\in L^{1}(G, w)$ and $\chi\in \widehat{G_{w}}$, then $\chi f\in L^{1}(G)$ and 
$$ 
\|\chi f\|_{1}\leq \|f\|_{1, w}
$$
(ii) If $\gamma\in \widehat G$ and $\chi\in \widehat{G_{w}}$, then $\gamma\chi\in \widehat{G_{w}}$.\\
(iii) The algebra $ L^{1}(G, w)$  is semisimple and so the Gelfand map on $L^{1}(G, w)$ is injective.
\end{lem}
\begin{proof}

(i) If $f\in L^{1}(G, w)$ and $\chi\in \widehat{G_{w}}$, then  
\begin{align*}
\|f\chi\|_{1} &= \int_{G}\mid f(x)\mid.\mid\chi(x)\mid dx\\
                     &\leq  \int_{G}\mid fw(x)\mid.\mid\frac{\chi}{w}(x)\mid dx\\
                      &= \|f\|_{1,w}< \infty
\end{align*}
Therefore $f\chi\in L^{1}(G)$.\\
(ii)  If $\gamma\in \widehat G$ and $\chi\in \widehat{G_{w}}$, then $\gamma$, and $\chi$ are multiplicative. So $\gamma\chi$  is multiplicative. Also
$$
\mid \frac{\gamma\chi}{w}(x)\mid= \mid\gamma\mid.\mid\frac{\gamma}{w}(x)\mid\leq 1
$$
This implies that $\gamma\chi\in \widehat{G_{w}}$.\\
(iii) If $f\in Rad(L^{1}(G, w))$, $\gamma\in \widehat G$ and $\chi\in \widehat G$, then $\chi\gamma(f)=0$. Assume that
\begin{align*}
\Phi : \widehat{G_{w}} &\to \Delta(L^{1}(G, w))\\
                                       & \chi\mapsto \Phi_{\chi}
\end{align*}
is homogeneous mapping, and 
\begin{align*}
\varphi : \widehat{G} &\to \Delta(L^{1}(G)\\
                                    & \gamma\mapsto \varphi_{\gamma}
\end{align*}
is homogeneous mapping such that 
$$
\Phi_{\chi}(f)= \int_{G}f(x)\overline{\chi(x)}dx
$$
and
$$
 \varphi_{\gamma}(g)= \int_{G} g(x)\overline{\gamma(x)}dx
$$
 for $f\in L^{1}(G, w)$ and $g\in L^{1}(G)$. This implies that 
$$
0= \Phi_{\gamma\chi}(f)= \varphi_{\gamma}(f\overline\chi)   \quad (\gamma\in \widehat G)
$$
 Since $L^{1}(G)$ is semisimple, then $f\overline\chi= 0$. Since $\overline\chi \neq 0$, so $f=0$. Therefore 
$$Rad(L^{1}(G, w))=\{0\}$$
 and $L^{1}(G, w)$ is semisimple.
\end{proof}
\begin{thm}
Let $G$ be an Abelian locally compact, and  $w: G\to (0, \infty)$ be Borel- measurable weighted function. Then the following expressions are equivalent: \\
(i) The Gelfand map $\Gamma_{w}: L^{1}(G, w) \to C_{0}({{\hat G}_{w}})$ is isometry.\\
(ii) The Gelfand map $\Gamma: L^{1}(G) \to C_{0}({{\hat G}})$ is isometry.\\
(iii)  $(L^{1}(G, w), \|.\|_{1,w})$ is a $C^{*}$- algebra.\\
(iv)  $(L^{1}(G), \|.\|_{1})$ is a $C^{*}$- algebra.\\
(v) $G$ is trivial group.
\end{thm}
\begin{proof}
 If  $\Gamma_{w}$  is isometry, then $\Gamma_{w}( L^{1}(G, w) )$ is a closed subalgebra of $C^{*}$- algebra $C_{0}({{\hat G}_{w}})$. Since $ L^{1}(G, w) $ is semisimple, so $\Gamma_{w}$ is injective.
Thus $ L^{1}(G, w) $ is $C^{*}$- algebra. As a result $ L^{1}(G, w) $ is regular and amenable and so due to \cite{RV},  $G$ ia s finite group. If $x\in G$, then 
\begin{align*}
w(e) &= \|\delta_{x}\delta_{x^{-1}}\|_{1,w}\\
         &= \|\delta_{x}^{2}\|_{1,w} ={w(x)}^{2}
\end{align*}
This implies that $w(e)= {w(e)}^{2}$ and $w(x) =1$ for all $x\in G$. Therefore 
$$ L^{1}(G, w) =  L^{1}(G)$$
 and the Gelfand map $\Gamma$ is isometry. Thus $ L^{1}(G)\cong {\mathbb C}^{n}$,  for some 
$n\in{\mathbb N}$. If 
$$x= (x_{1}, \cdots, x_{n}) \in {\mathbb C}^{n}$$
 then 
$$ 
\|x\|_{1} = \|x\|_{\infty}
$$
and so
$$
\mid x_{1}\mid+\cdots +\mid x_{n}\mid= max\{\mid x_{i}\mid: 1\leq i\leq n\}
$$
Therefore there exist $1\leq j\leq n$, such that 
$$
\|x\|_{1}= \mid x\mid_{j}
$$
So $x_{i}=0$ for $i\neq j$.  Therefore $G$ is trivial group and $L^{1}(G)\cong {\mathbb C}$.
\end{proof}
\begin{thm}
Let $G$ be a locally compact Abelian group and $w: G\to (0, \infty)$ be a Borel measurable weighted function. Then the following statements are equivalent\\
(i) $\Gamma_{w}$ is surjective.\\
(ii) $L^{1}(G)= C_{0}(\hat G)$.\\
(iii) $G$ is finite.
\end{thm}
\begin{proof}
If $L^{1}(G, w)= C_{0}(\widehat {G_{w}})$, then by the open- mapping  theorem, there exist $M>0$ such that 
$$
\|\hat f\|_{\infty}\leq \|f\|_{1, w}\leq M\|\hat f\|_{\infty}
$$
for all $f\in L^{1}(G, w)$. Thus 
$$
L^{1}(G, w)\cong C_{0}(\widehat {G_{w}})
$$
with equivalent norms. But every Abelian $C^{*}$- algebra is Arens regular and amenable. Therefore $L^{1}(G, w)$ is Arens- regular and amenable Banach algebra. 
Since $G$ is Abelian, $G$ is amenable. Hence $w^{*}$ is bounded and $\Omega$ is bounded below for some $m>0$.  Thus $\Omega$ is not 0- cluster. Consequently, $G$ is finite; see \cite{RV1}. Also
$$
L^{1}(G, w)\cong L^{1}(G)
$$
with equivalent norms, so $\widehat{G_{w}}= \hat G$. Hence 
$$
C_{0}(\widehat{G_{w}})= C_{0}(\hat G)
$$
Thus $ L^{1}(G)= C_{0}(\hat G)$. This completes the proof.

\end{proof}
\begin{rem}
Let $A$ be a commutative semisimple Banach algebra and $\Gamma_{A}$ be surjective. Then 
$$
A\cong C_{0}(\Delta(A))
$$
with equivalent norms and $A$ is Arens- regular and amenable Banach algebra. Also, 
$$
A^{**}\cong C(X)
$$
where $X:= \Delta(\Delta(A^{**}))$. Hence $A^{**}$ is Arens- regular and amenable Banach algebra.
\end{rem}

\begin{ex}
Let $G= ({\mathbb Z},+)$ and $w(n)= 1+ \mid n\mid$, for $n\in\mathbb Z$.  Then\\
(i) 
$$
\underset{n\to\infty}{lim} w(n+ m)^{\frac{1}{n}}= 1
$$
for all $m\in\mathbb Z$.\\
(ii)
$$
\sum_{n=1}^{\infty}\frac{ Lnw(n+ m)}{n^{2}}<\infty
$$
and so $w$ is nonquasianalytic.\\
(iii)
$$
\Delta(l_{1}({\mathbb Z}, w))= \mathbb T
$$
(iv) $l_{1}(\mathbb{Z}, w)$ is regular and Tauberian.
\end{ex}
\begin{proof}
(i)
If 
$$
k_{n}:= (1+\mid n+ m\mid)^{\frac{1}{n}}-1
$$
then
$$
1+ \mid n+ m\mid= (1+ k_{n})^{n}\geq 1+\frac{n(n-1)}{2}k_{n}^{2}
$$
and so 
$$
\underset{n\to \infty}{lim}k_{n}=0
$$
Thus $r_{w}(m)= 1$.\\

(ii) If $a_{n} = \frac{Ln(m+ n)}{n^{2}}$ and $b_{n}= {\frac{1}{n^{\frac{3}{2}}}}$, then 
$$
\underset{n\to \infty}{lim} \frac{a_{n}}{b_{n}}=0
$$
and
$$
\sum_{n=1}^{\infty} b_{n}
$$
is convergent . Thus 
$$
\sum_{n=1}^{\infty} a_{n}
$$
is convergent , for all $m$, and $w$ is nonquasianalytic.\\
(iii)
Due to [\cite{kan}, 2.8.8],
 the following is yield:
\begin{align*}
\Delta(l_{1}(\mathbb{Z}, w))= \{z\in\mathbb C: R_{-}\leq \mid z\mid\leq R_{+}\}
\end{align*}
Where
$$
R_{+}= inf\{w(n)^{\frac{1}{n}}: n\in \mathbb N\}=1
$$
and 
$$
R_{-}= sup\{w(-n)^{\frac{-1}{n}}: n\in \mathbb N\}=1
$$
Therefore
$$
\Delta(l_{1}(\mathbb{Z}, w))= \{z\in\mathbb C:  \mid z\mid=1\}= \mathbb T
$$
such that 
$$
\varphi_{z}(f)= \sum_{n=-\infty}^{+\infty}f(n)z^{n}
$$
for $f\in l_{1}(\mathbb{Z}, w)$ and $z\in \mathbb T$.\\
(iv)
According to Theorem \ref{thrd}, $l_{1}(\mathbb{Z}, w)$ is regular. Hence By applying Lemma \ref{lT}, it is Tauberian.

\end{proof}


\begin{thebibliography}{9999}

\bibitem{D}
H. G. Dales, Banach Algebras and Automatic Continuity. London Math. Soc. Monogr. 24, Clarendon Press, Oxford,2000.
\bibitem{DL}
 H. G. Dales and A. T. Lau, The second duals of Beurling algebras, Mem. Amer. Math. Soc., 177
(836) (2005).

\bibitem{Do}
Y. Domar, Harmonic Analysis on certain commutative Banach algebra, Acta Math. 96: 1-66 (1956).

\bibitem{GRS}
I. Gelfand, D. Raikov and G. Shilov, Commutative normed rings, AMS Chelsea Publishing, 1964.
\bibitem{Gr}
N. Gronback, Amenability of weighted convolution algebras on locally compact groups, Trans.
Amer. Math. Soc., 319 (1990) 765–775.

\bibitem{kan}
E. Kaniuth, A Course in Commutative Banach Algebra, Springer Science, Bussiness Media, LLC 2009.
\bibitem{LN}
K. B. Laursen and M. M. Neumann, An introduction to local spectral theory, Clarendon
Press, Oxford, (2000).
\bibitem{MR}
S. Maghsoudi and A. Rejali, Unbounded weighted Radon measures and dual of certain function
spaces with strict topology, Bull. Malays. Math. Sci. Soc., 36 (1) (2013) 211–219
\bibitem{MR2}
S. Maghsoudi and A. Rejali, On the dual of certain locally convex function spaces, Bull. Iranian
Math. Soc., 41 (4) (2015) 1003–1017.
\bibitem{MER}
M. J. Mehdipour and A. Rejali, Regularity and amenability of weighted Banach algebras and
their second dual on locally compact groups, arXiv:2112.13286v1, (2022).
\bibitem{MER2}
M. J. Mehdipour and A. Rejali, weak amenability of weighted group algebras, arXiv: 2209.08346v1, (2022).
\bibitem{MER3}
M. J. Mehdipour and A. Rejali, weak amenability of weighted measure algebras and
their second dual, arXiv:2210. 04075v1, (2022).
\bibitem{R}
A. Rejali, The analogue of weighted group algebra for semitopological semigroups, J. Sci. Islam.
Repub. Iran, 6 (2) (1995) 113–120
\bibitem{RV1}
A. Rejali and H. R. E.  Vishki,  Regularity and amenability of the second dual of weighted group algebras, Proyecciones, Vol. 26, pp. 259- 267, (2007).
\bibitem{RV}
A. Rejali and H. R. Vishki, Weighted convolution measure algebras characterized by convolution
algebras, J. Sci. Islam. Repub. Iran, 19 (2) (2008) 169–173.

\bibitem{PR}
P. Ressel and W. J. Ricker, Vector-valued positive definite functions, the Berg- Maserick theorem, and applications, Math. Scand., 90, (2002), 289- 319.




%
 



\end{thebibliography}
\end{document}